\documentclass[12pt]{amsart}
\usepackage{epsf,overpic,amssymb,amscd,stmaryrd}

\newtheorem{thm}{Theorem}[section]

\newtheorem{lemma}[thm]{Lemma}
\newtheorem{cor}[thm]{Corollary}

\numberwithin{equation}{subsection}

\newcommand{\R}{\mathbb{R}}
\newcommand{\Z}{\mathbb{Z}}
\newcommand{\Q}{\mathbb{Q}}
\newcommand{\N}{\mathbb{N}}

\newcommand{\bdry}{\partial}
\newcommand{\s}{\vskip.1in}
\newcommand{\n}{\noindent}

\newcommand{\be}{\begin{enumerate}}
\newcommand{\ee}{\end{enumerate}}

\begin{document}

\title{Stabilizing the monodromy of an open book decomposition}

\author{Vincent Colin}
\address{Universit\'e de Nantes, UMR 6629 du CNRS, 44322 Nantes, France}
\email{Vincent.Colin@math.univ-nantes.fr}

\author{Ko Honda}
\address{University of Southern California, Los Angeles, CA 90089}
\email{khonda@math.usc.edu} \urladdr{http://rcf.usc.edu/\char126
khonda}

\date{This version: June 5, 2007.}

\thanks{KH supported by an Alfred P.\ Sloan Fellowship and an NSF CAREER Award (DMS-0237386).}

\begin{abstract}
We prove that any mapping class on a compact oriented surface with
nonempty boundary can be made pseudo-Anosov and right-veering after
a sequence of positive stabilizations.
\end{abstract}

\maketitle

\section{Introduction}

Let $S$ be a compact oriented surface with nonempty, possibly
disconnected, boundary $\bdry S$, and let $h:S\rightarrow S$ be a
diffeomorphism for which $h|_{\bdry S}=id$. By a theorem of
Giroux~\cite{Gi}, there is a 1-1 correspondence between isomorphism
classes of contact structures on closed 3-manifolds and equivalence
classes of pairs $(S,h)$, up to {\em positive stabilization} and
{\em conjugation}. The pair $(S,h)$ can be interpreted as an {\em
open book decomposition} of some 3-manifold $M$, as will be
explained later on in the introduction.

A {\it positive stabilization} of a pair $(S,h)$ is defined as
follows. Let $S'$ be the oriented union of the surface $S$ and a
band $B$ attached along the boundary of $S$, i.e., $S'$ is obtained
from $S$ by attaching a 1-handle along $\bdry S$.  Let $\gamma$ be a
simple closed curve in $S'$ which intersects the co-core of $B$ at
exactly one point and let $id_B\cup h$ be the extension of $h$ by
the identity map to $S'$. Then define $h'= R_\gamma\circ (id_B\cup
h)$, where $R_\gamma$ is a positive Dehn twist about $\gamma$. The
pair $(S',h')$ is called an {\em elementary positive stabilization}
of $(S,h)$. More generally, we say that $(S',h')$ is a {\em positive
stabilization} of $(S,h)$ if it is obtained from $(S,h)$ by a
sequence of elementary positive stabilizations. In this paper, a
{\em stabilization} will always mean ``positive stabilization''. By
{\em conjugation} we mean replacing $(S,h)$ by $(S, ghg^{-1})$,
where $g$ is a diffeomorphism of $S$ which is not necessarily the
identity on the boundary.

Next we describe the notion of a {\em right-veering} diffeomorphism
$h : S\to S$, introduced in \cite{HKM}.  Let $\alpha$ and $\beta$ be
properly embedded oriented arcs $[0,1]\rightarrow S$ with a common
initial point $x\in\bdry S$. Denote by $\pi :\tilde{S} \to S$ the
universal cover of $S$. We pick lifts $\tilde{\alpha}$ and
$\tilde{\beta}$ of $\alpha$ and $\beta$ in $\tilde{S}$, starting at
the same lift $\tilde{x}$ of $x$. The arc $\tilde{\alpha}$ divides
$\tilde{S}$ into two regions, one to the left and the other to the
right of $\tilde{\alpha}$; the left-right convention is such that at
$\tilde{x}$ the (tangential) orientation of $\partial \tilde{S}$
points towards the right-hand component of $\tilde{S}\setminus
\tilde \alpha$. We say that $\beta$ is {\em to the right of
$\alpha$} if $\tilde{\beta} (1)$ belongs to the region to the right
of $\tilde{\alpha}$ or equals $\tilde{\alpha} (1)$. (Observe that
this definition clearly does not depend on the choice of
representatives in the isotopy classes of $\alpha$ and $\beta$, rel
boundary.) Now, if $h$ is a diffeomorphism of a compact surface $S$
with $\partial S\neq \emptyset$ and $h\vert_{\partial S}=id$, we say
that $h$ is {\it right-veering} if for every properly embedded
oriented arc $\alpha$ in $S$, the arc $h(\alpha )$ is to the right
of $\alpha$. A typical example of a right-veering diffeomorphism is
a positive ($=$ right-handed) Dehn twist around a simple closed
curve in $S$.

The goal of this note is to prove the following:

\begin{thm}\label{thm:stabilization}
Let $S$ be a compact oriented surface with nonempty boundary and $h$
be a diffeomorphism of $S$ which is the identity on $\partial S$.
Then there exists a stabilization $(S',h')$ of $(S,h)$, where $\bdry
S'$ is connected and $h'$ is right-veering and freely homotopic to a
pseudo-Anosov homeomorphism.
\end{thm}

Let us briefly recall the $3$-dimensional context for
Theorem~\ref{thm:stabilization}, which is our primary motivation. An
oriented, positive {\it contact structure} on a closed oriented
$3$-manifold $M$ is a plane field $\xi$ given as the kernel of a
global $1$-form $\alpha$, whose exterior product with $d\alpha$ is a
volume form on $M$. An {\it open book decomposition} of $M$ is a
pair $(K,\theta )$, where $K$ is a link in $M$ and $\theta :
M\setminus K \to S^1$ is a fibration given by $(k,r,\phi )\mapsto
\phi$ in a neighborhood $N(K)\simeq \{ (k,r,\phi) \in S^1 \times
[0,1)\times S^1 \}$ ($(r,\phi)$ are polar coordinates) of $K\simeq
\{ k\in S^1 , r=0\}$. The link $K$ is called the {\it binding} and
the fibers of $\theta$ are called the {\it pages} of the open book.
Notice that the orientations given on $M$ and $S^1$ induce a
co-orientation of the pages and thus an orientation of $K$ as the
boundary of the closure of one page. An open book decomposition of
$M$ is completely determined, up to a diffeomorphism of $M$, by a
compact retraction $S$ of a  page, and a monodromy map $h:S\to S$
for $\theta$ which is the identity on $\partial S$. Conversely, any
pair $(S,h)$ gives rise to an open book decomposition of some
manifold, called the {\it relative suspension} of $(S,h)$. Moreover,
the open book decompositions given by $(S,h)$ and a stabilization
$(S',h')$ of $(S,h)$ are conjugated by a diffeomorphism which
induces the natural inclusion $i : S\to S'$ on one page.

The link between contact structures and open books can be expressed
via the following definition, introduced by Giroux. A contact
structure $\xi$ is {\it carried} by an open book $(K,\theta )$ if
there exists a $1$-form $\alpha$ with kernel $\xi$, such that
$d\alpha$ gives an area form on the pages and $\alpha$ a length form
on $K$. The main theorem of Giroux~\cite{Gi} states that every
contact structure $\xi$ on a closed oriented $3$-manifold $M$ is
carried by an open book, and that two open books of $M$ carry
isotopic contact structures if and only if they have isotopic
stabilizations. We then have the following immediate translation of
Theorem~\ref{thm:stabilization}:

\begin{cor}
On a closed oriented $3$-manifold $M$, every oriented, positive
contact structure is carried by an open book whose binding is
connected, and whose monodromy is right-veering and freely homotopic
to a pseudo-Anosov homeomorphism.
\end{cor}

Although the primary motivation of this paper is $3$-dimensional,
our perspective, as well as the proof of
Theorem~\ref{thm:stabilization}, will mostly be 2-dimensional.

\section{The curve complex of $S$}

The proof of Theorem~\ref{thm:stabilization} uses distances in the
{\em curve complex} $\mathcal{C}(S)$, introduced by
Harvey~\cite{Har}, as well as its hyperbolicity properties,
due to Masur and Minsky~\cite{MM}, in an essential way. In this section we briefly
review the necessary background on the 1-skeleton of the curve
complex $\mathcal{C}(S)$.  We assume the reader is familiar with the
basics of the Nielsen-Thurston theory of surface homeomorphisms.
(See, for example,  \cite{Bo, CB, FLP, Th}.)

Suppose the genus $g(S)>1$. Since we are free to stabilize in
Theorem~\ref{thm:stabilization}, this is not a serious restriction.
(When $g(S)=1$ and $\bdry S$ is connected, the definition of
$\mathcal{C}(S)$ is slightly different.) The vertices of
$\mathcal{C}(S)$ are isotopy classes of non-peripheral (i.e., not
isotopic to a component of $\bdry S$) simple closed curves on $S$.
There is an edge of length one connecting each pair
$\{\alpha,\beta\}$ of vertices, if $\alpha$ and $\beta$ are distinct
isotopy classes which can be realized disjointly, i.e., have
geometric intersection $i(\alpha,\beta)=0$. Denote the $k$-skeleton
of the curve complex by $\mathcal{C}_k(S)$. The 1-skeleton
$\mathcal{C}_1(S)$ is a geodesic metric space, and the distance
between $\alpha,\beta\in \mathcal{C}_0(S)$ is denoted by
$d_{\mathcal{C}(S)}(\alpha,\beta)$, or simply by $d(\alpha,\beta)$.
(We have no need for the higher-dimensional simplices of
$\mathcal{C}(S)$ in this paper.)

The following useful facts can be found in \cite{MM}:

\begin{lemma} \label{lemma:intersection}

\be
\item If $\alpha, \beta\in \mathcal{C}_0(S)$, then
$d(\alpha,\beta)\leq 2i(\alpha,\beta)+1$.

\item If $d(\alpha,\beta)\geq 3$, then $\alpha$ and $\beta$ {\em
fill} $S$, i.e., any non-peripheral simple closed curve $\gamma$
must intersect $\alpha$ or $\beta$.

\ee
\end{lemma}

We learned the following from Yair Minsky, at least for the case
when $h$ is pseudo-Anosov:

\begin{lemma}\label{lemma:findgamma}
For any $h\in Map(S,\bdry S)$ which is not freely isotopic to a
periodic diffeomorphism, there exists $\alpha\in \mathcal{C}_0(S)$
such that $d(\alpha,h(\alpha))$ is arbitrarily large.
\end{lemma}

\begin{proof}
Fix a reference hyperbolic metric on $S$ with geodesic boundary. A
geodesic lamination $\mu$ on $S$ is {\em minimal} if it does not
contain any proper sublamination. In particular, $\mu$ does not have
any closed leaves, and components of $\bdry S$ cannot be leaves of
$\mu$. The lamination $\mu$ is {\em filling} if each component of
$S-\mu$ (called a {\em complementary region}) is an ideal polygon or
a ``once-punctured ideal polygon''. (Strictly speaking, the latter
is a half-open annulus, whose boundary component is a component of
$\bdry S$.)  Let $\mathcal{EL}(S)$ be the set of minimal filling
laminations, viewed as a subset of
$\mathcal{ML}(S)/\mbox{measures}$, where $\mathcal{ML}(S)$ is the
space of measured geodesic laminations. (A {\em measured geodesic
lamination} is defined to be a compact geodesic lamination
$\lambda$, which is endowed with a transverse measure whose support
is {\em all of} $\lambda$.  In particular, $\lambda$ has no infinite
isolated leaf.  An example of such a leaf is a diagonal $l$ of an
ideal $n$-gon complementary region of a lamination $\lambda$, with
$n\geq 4$.)

We claim that there is a minimal filling geodesic lamination
$\mu\subset S$ so that $h(\mu)\not=\mu$.
If $h$ is (freely homotopic to) a pseudo-Anosov homeomorphism, then
we simply pick a minimal filling $\mu$ which is not the stable
lamination or the unstable lamination of $h$. If $h$ is reducible,
then, after taking a sufficiently large power $h^n$ of $h$, we may
assume that there is a collection of disjoint homotopically
nontrivial annuli $A_1,\dots,A_m$ so that $h^n$ on each component
$S_j$ of $S-\cup_i A_i$ is either the identity or pseudo-Anosov, and
$h^n$ on $A_i$ is $R_{\gamma_i}^{n_i}$, $n_i\in \Z$, where
$\gamma_i$ is the core curve of $A_i$.  In this case, any minimal
filling lamination $\mu$ with ideal 3-gon and once-punctured ideal
monogon complementary regions would work, as we explain in the next
two paragraphs.

Suppose there is a pseudo-Anosov piece $S_j$. The restriction of
$\mu$ to $S_j$ consists of a finite number of disjoint non-parallel
arc types which cut up $S_j$ into 3-gons and once-punctured
monogons. By a slight abuse of notation, let $\mu|_{S_j}$ denote the
union of arcs, one from each arc type. We then claim that the
minimum geometric intersection $i(h(\mu)|_{S_j},\mu|_{S_j})\not=0$,
where the endpoints of the arcs are free to move around $\bdry S_j$.
Suppose $i(h(\mu)|_{S_j},\mu|_{S_j})=0$. Since $S_j-\mu|_{S_j}$
consists of 3-gons and once-punctured monogons, it follows that
$\mu|_{S_j}=h(\mu)|_{S_j}$, and $h$ must be periodic on $S_j$, a
contradiction. Thus, $\mu$ and $h(\mu)$ will have nontrivial
intersection, and are distinct.

If there are no pseudo-Anosov components, then $h^n$ consists of
disjoint (high multiples of) Dehn twists.  One can also verify that
$\mu$ and $h(\mu)$ nontrivially intersect in this case. Indeed, let
$\overline{S}_j$ be the union of $S_j$ and all its adjacent annuli
$A_i$.  For simplicity, assume that no $A_i$ bounds $S_j$ along both
boundary components.  Let $\delta$ be a properly embedded arc of
$\overline{S}_j$ from $A_i$ to another $A_{i'}$, and let $\delta'$
be a parallel push-off of $\delta$.  Then $h(\delta')$ and $\delta$
intersect efficiently if $n_i$ and $n_{i'}$ have the same sign; if
they have opposite signs, $h(\delta')$ and $\delta$ intersect
efficiently for a push-off to one side and have two extraneous
intersections for a push-off to the other side.  In either case,
$\delta$ and $h(\delta')$ intersect nontrivially, provided
$|n_i|+|n_{i'}|\gg 0$.  The same argument holds for $\mu$.

Next, for any $\mu$ which is minimal and filling, we claim there
exists a sequence of (connected) simple closed curves which
converges to $\mu$. Pick a point $p\in \mu$ and an arbitrarily short
transversal $\delta$ to $\mu$ so that $p\in int(\delta)$. Take a
neighborhood $N(p)$ of $p$ of the form $[-1,1]\times [-1,1]$, where
leaves of $\mu\cap N(p)$ are $[-1,1]\times\{pt\}$,
$\delta=\{0\}\times [-1,1]$, and $p=(0,0)$. Let $\beta$ be an arc
which starts at $p$ and follows along $\mu$ in one direction, until
the first return to $\delta$. If $\mu$ is orientable, then we can
close up $\beta$ and ``stretch it tight'' to obtain a closed curve
which is arbitrarily close to $\mu$. If $\mu$ is not orientable,
then the arc $\beta$ could return to $\delta$ at $p_1$ from the
``same side'', i.e., the same component of
$N(p)-(\{0\}\times[-1,1])$, from which it left $p$. Next consider
the subarc $\delta_1\subset \delta=\delta_0$ which contains $p$ and
has $p_1$ as an endpoint. Continuing $\beta$ past $p_1$, let $p_2$
be the first return to $\delta_1$. If the approach is from the
``same side'' again, then we take a shorter $\delta_2\subset
\delta_1$, and continue.  If further closest returns to $\delta$ are
always from the ``same side'', then we can naturally orient $\mu$ at
the point $p$. If we can do this at every point on $\mu$, then $\mu$
would be orientable, a contradiction.

Masur and Minsky~\cite{MM} have shown that the curve complex
$\mathcal{C}(S)$ is hyperbolic in the sense of Gromov; hence it
makes sense to talk about its boundary at infinity $\bdry_\infty
\mathcal{C}(S)$. Now, according to a theorem of Klarreich~\cite{Kl},
there is a homeomorphism between $\mathcal{EL}(S)$ and $\bdry_\infty
\mathcal{C}(S)$. Moreover, a sequence $\beta_i\in \mathcal{C}_0(S)$
converges to $\beta\in\bdry_\infty\mathcal{C}(S)$ if and only if it
converges to $\beta\in \mathcal{EL}(S)$ in the topology of
$\mathcal{ML}(S)/\mbox{measures}$.  (Klarreich's theorem, in turn,
relies on the results of Masur-Minsky \cite{MM}.) Let $\mu$ be a
minimal filling geodesic lamination so that $h(\mu)\not= \mu$, and
let $\alpha_n$, $n=1,2,\dots$, be a sequence of simple closed curves
which converges to $\mu$ in the topology of
$\mathcal{ML}(S)/\mbox{measures}$.  Clearly $h(\alpha_n)\rightarrow
h(\mu)$. Since $\mu$ and $h(\mu)$ are distinct points on
$\bdry_\infty \mathcal{C}(S)$, it immediately follows that
$d(\alpha_n,h(\alpha_n))\rightarrow \infty$.

Here is a more direct proof, which the authors learned from Yair
Minsky. The technique apparently can be traced back to
Kobayashi~\cite{Ko}, and has been used by Hempel~\cite{He} and
Abrams-Schleimer \cite{AS} to study distances of Heegaard
splittings. Suppose $d(\alpha_i,h(\alpha_i))$ does not approach
$\infty$. Then, after passing to subsequences, we may assume that
$d(\alpha_i,h(\alpha_i)) = N$ for a constant $N$. Consider a
geodesic in the curve complex which connects $\alpha_i$ to
$h(\alpha_i)$ --- write it as $\alpha_i, A_{i,1}, A_{i,2},\dots,
A_{i,N}=h(\alpha_i)$.  Then $A_{i,1}$ is disjoint from $\alpha_i$,
and hence, after taking a subsequence, $A_{i,1}$ converges to a
lamination $\mu'$ which has zero geometric intersection with $\mu$.
Similarly, $A_{i,{N-1}}$ converges to a lamination $\nu$ which has
zero geometric intersection with $h(\mu)$. Since $\mu$ is minimal,
we may conclude that $\mu'=\mu$ and $\nu=h(\mu)$. We now have a new
pair of sequences converging to $\mu$ and $h(\mu)$, but at a shorter
distance. By repeating the procedure, we eventually conclude that
$\mu=h(\mu)$, a contradiction.
\end{proof}

\section{Proof of Theorem~\ref{thm:stabilization}}

Using the technique in (\cite{HKM}, Proposition~6.1), any monodromy
map $(S,h)$ may be stabilized so that it becomes right-veering. In
the proof, a four-times punctured sphere is attached onto a boundary
component of $\bdry S$.  This construction in \cite{HKM} makes the
monodromy map reducible, so we may additionally assume that $h$ is
not periodic. Observe that, once $(S,h)$ is right-veering,
stabilizations of it will remain right-veering.

Throughout the proof, $d$ will always mean $d_{\mathcal{C}(S)}$. We
will write $i_\Sigma(\cdot,\cdot)$ to indicate the geometric
intersection number on the surface $\Sigma$.

\s\n {\bf Case 1: $\bdry S$ is connected.}  Suppose that $\bdry S$
is connected. By Lemma~\ref{lemma:findgamma}, there exists a
non-peripheral $\gamma_0\in \mathcal{C}_0(S)$ satisfying
$d(h(\gamma_0),\gamma_0)=N\gg 0$. Let $\gamma$ be a properly
embedded arc in $S$ which becomes $\gamma_0$ after concatenating
with an arc of $\bdry S$ which connects the endpoints of $\gamma$.
Stabilize along the arc $\gamma$ to obtain $(S',h'=R_{\gamma'}\circ
h)$.  Here $\gamma'$ is the extension of $\gamma$ to $S'$ so that
$\gamma'$ intersects the co-core $a$ of the 1-handle at exactly one
point.  We will show that $h'$ is pseudo-Anosov.  Although $\bdry
S'$ will have two boundary components, this will be remedied in Case
2.

We first prove that $h'(\delta)\not=\delta$ for any multicurve
(i.e., closed embedded 1-manifold) $\delta$ of $S'$ without
peripheral components. This would imply that $h'$ is not reducible.
The basic idea of the proof is to distinguish $h'(\delta)$ and
$\delta$ by intersecting with the co-core $a$ and the closed curve
$\gamma'$.

First suppose that $\delta\subset S$, i.e., $\delta$ does not
intersect the co-core $a$. If $i_{S'}(h(\delta),\gamma')=m$, then we
claim that $i_{S'}(R_{\gamma'}\circ h(\delta),a)=m$. (Note that
$i_{S'}(\alpha,\gamma')=i_S(\alpha,\gamma)$ if $\alpha$ is a closed
curve on $S$.) Clearly, there is a representative $g$ of
$R_{\gamma'}\circ h(\delta)$ which intersects $a$ at $m$ points. If
$i_{S'}(R_{\gamma'}\circ h(\delta),a)<m$, then there must exist a
bigon consisting of a subarc of $a$ and an arc of $g$. It follows
(without too much difficulty) that there is a bigon consisting of an
arc of $\gamma$ and an arc of $h(\delta)$, a contradiction. Since
$i_{S'}(\delta,a)=0$ and $i_{S'}(h'(\delta),a)=m$, we have
$h'(\delta)\not=\delta$ if $m>0$. Observe that this case covers the
possibility that a component of $\delta$ is $\bdry S$. On the other
hand, if $m=0$, then $i_S(h(\alpha),\gamma_0)=0$ and
$d(h(\alpha),\gamma_0)=1$ for every component $\alpha$ of $\delta$.
Since $d(h(\gamma_0),\gamma_0)=N$, it follows that
$d(h(\alpha),h(\gamma_0))\approx N$ for every $\alpha$. (Here
$\approx$ means ``approximately equal to''.) Composing with
$h^{-1}$, we obtain $d(\alpha,\gamma_0)\approx N$. With the help of
Lemma~\ref{lemma:intersection}, we find that
$i_S(\alpha,\gamma_0)\gtrapprox {N-1\over 2}$.  Hence
$i_{S}(\alpha,\gamma)=i_{S'}(\alpha,\gamma')\gtrapprox {N-1\over 2}$
for each component $\alpha$ of $\delta$. Since $m=0$,
$R_{\gamma'}\circ h(\delta)=h(\delta)$ and
$i_{S'}(h'(\delta),\gamma')=0$. By comparing
$i_{S'}(\cdot,\gamma')$, we obtain $h'(\delta)\not = \delta$ for
$m=0$.

Next suppose that $\delta\not\subset S$.  Let $k=i_{S'}(\delta,a)$.
Write $B=S'-int(S)=[-1,1]\times[-1,1]$, so that $\{\pm 1\}\times
[-1,1]\subset \bdry S'$, $a=[-1,1]\times \{0\}$, and $\gamma'\cap
B=\{0\}\times[-1,1]$. We now explain how to normalize $\delta$.
Isotop $\delta$ so that it intersects $\gamma'$ and $a$ transversely
and efficiently, for example by realizing $\gamma'$, $a$, and
$\delta$ as geodesics.   Then subdivide $\delta$ into arcs
$\delta_1,\dots,\delta_k,\delta_1',\dots,\delta_k'$, where
$\delta_i\subset S$ and $\delta_i'\subset B$.  (The $\delta_i$ and
$\delta_i'$ are not ordered in any particular way.) The $\delta_i'$
are linear arcs in $B$, each with an endpoint on
$[-1,1]\times\{-1\}$ and on $[-1,1]\times\{1\}$. If $\delta_i'$ does
not intersect $\gamma'\cap B$, then we assume that $\delta_i'$ is
vertical, i.e., $\{pt\}\times[-1,1]$. Moreover, we may normalize
$\delta$ so that there is no triangle in $S$ whose boundary consists
of (i) a subarc of $\delta_i$, (ii) a subarc of $\gamma$, and (iii)
a subarc of $[-1,1]\times\{\pm 1\}$. If there is such a triangle,
then we can isotop $\delta$ and push the triangle into $B$.  Note
that if $\delta_i$ is parallel to $\gamma$, then the isotopy may not
be unique.  We similarly normalize $h(\delta)$ and subdivide
$h(\delta)$ into arcs
$(h(\delta))_1,\dots,(h(\delta))_k,(h(\delta))_1',\dots,(h(\delta))_k'$.

We claim that if any $\delta_i$ is a boundary-parallel arc in $S$,
then $\delta$ must have a component which is parallel to $\bdry S'$.
Indeed, $\delta_i$ either has endpoints on both components of
$[-1,1]\times\{\pm 1\}$ or begins and ends on a single component
(but does not form a bigon together with some arc of $[-1,1]\times
\{\pm 1\}$). In the former case, if the endpoints of $\delta_i$ are
not connected by a single $\delta_j'$, then $\delta$ will be
spiraling towards one component of $\bdry S'$, a contradiction.  In
the latter case, one of the two $\delta_{j_1}'$, $\delta_{j_2}'$
which begin at the endpoints of $\delta_i$ continues on to spiral
around one component of $\bdry S'$, also a contradiction. Since we
are assuming that $\delta$ has no peripheral components, it follows
that no $\delta_i$ is parallel to $\bdry S$.

Consider the case where some $(h(\delta))_i'$ has negative slope
$\not=-\infty$. Let $m>0$ be the number of arcs with negative slope.
The rest of the arcs $(h(\delta))_i'$ will have slope $\infty$. Also
let $n=\sum_{j=1}^k i_S((h(\delta))_j,\gamma)$, i.e., the number of
intersections between $h(\delta)$ and $\gamma$, away from $B$.
Figure~\ref{negative} depicts this situation.
\begin{figure}[ht]
\s
\begin{overpic}[height=7cm]{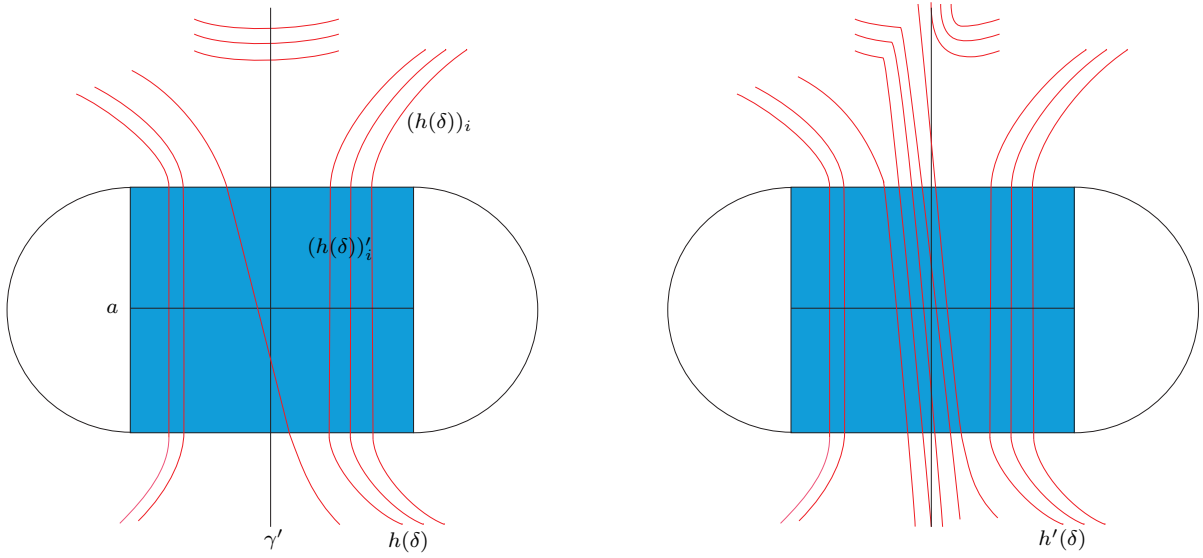}
\put(21.6,-1.5){\tiny $\gamma'$} \put(32,-1.6){\tiny $h(\delta)$}
\put(86.5,-1.5){\tiny $h'(\delta)$} \put(8.4,18){\tiny $a$}
\put(25.2,23){\tiny $(h(\delta))_i'$} \put(33.5,33.5){\tiny
$(h(\delta))_i$}
\end{overpic}
\s \caption{The case where some $(h(\delta))_i'$ have negative
slope. The shaded region is the $1$-handle $B$. In the picture,
$k=6$, $m=1$, and $n=3$ for $h(\delta)$.  The right-hand diagram
depicts the effect of $R_{\gamma'}$ on $h(\delta)$.}
\label{negative}
\end{figure}
We then claim that $i_{S'}(h'(\delta),a)=i_{S'}(R_{\gamma'}\circ
h(\delta),a)= k+ m +n$. Since $i_{S'}(\delta,a)=k$, this would show
that $\delta\not = h'(\delta)$. The representative of $h'(\delta)$
shown in Figure~\ref{negative} intersects $a$ at $k+m+n$ points, and
we need to show that the intersection is efficient.  In other words,
no subarc of $(h'(\delta))_i$ bounds a bigon together with a subarc
of $a$. To prove the claim, consider the subsurface $S''\subset S'$
which is the union of $B$ and a small neighborhood of $\gamma'$
containing the support of the Dehn twist $R_{\gamma'}$. The
intersection of our representative of $h'(\delta)$ with $S''$
consists of non-boundary-parallel arcs with the exception of
vertical arcs in $B$.  Observe that the restriction of $h'(\delta)$
to $S'-S''$ is the same as the restriction of $h(\delta)$ to
$S'-S''$. Hence, it suffices to prove that there is no
boundary-parallel component of $(S'-S'')\cap h(\delta)$ which
cobounds a bigon together with a subarc of $\bdry S''-\bdry S'$.
Such a boundary-parallel arc contradicts our normalization.

Next consider the case where no $(h(\delta))_i'$ has negative slope.
Let $m$ be the number of arcs $(h(\delta))_i'$ with positive slope
and let $n=\sum_j i_{S}((h(\delta))_j,\gamma)$, as before. We have
$i_{S'}(h'(\delta),a)=k-m+n$, which is precisely the number of
intersection points between our particular representative of
$h'(\delta)$ and $a$. (The proof is the same as in the previous
case.) If $m\not=n$, then
$i_{S'}(\delta,a)\not=i_{S'}(h'(\delta),a)$. It remains to consider
the case $m=n$. In this case, $i_{S'}(h'(\delta),\gamma')=
m+n=2m\leq 2k$.  We argue by contradiction that
$i_{S'}(\delta,\gamma')\not=i_{S'}(h'(\delta),\gamma')$. Suppose
$i_{S'}(\delta,\gamma')=i_{S'}(h'(\delta),\gamma')\leq 2k$. Since
$1\leq i\leq k$, this means that there is some $i$ for which
$i_S(\delta_i,\gamma)\leq {2k\over k}=2$. If we close up $\delta_i$
by concatenating with an arc of $\bdry S$ to obtain the simple
closed curve $\overline{\delta}_i\subset S$, then
$i_S(\overline{\delta}_i,\gamma_0)\leq 3$. (Recall that $\delta_i$
is not boundary-parallel on $S$, so $\overline{\delta}_i$ is neither
homotopically trivial nor parallel to $\bdry S$.) By
Lemma~\ref{lemma:intersection}, $d(\overline{\delta}_i,\gamma_0)\leq
7$.  For simplicity, we write
$d(\overline{\delta}_i,\gamma_0)\approx 0$. Since $\delta_i$ and
$\delta_j$ are disjoint, it is not hard to see that
$d(\overline{\delta}_i,\overline{\delta}_j)\approx 0$ for all $i,j$.
Hence $d(\overline{\delta}_j,\gamma_0)\approx 0$ for any $j$. Acting
by $h$, we have $d(h(\overline{\delta}_j),h(\gamma_0))\approx 0$.
Since $d(h(\gamma_0),\gamma_0)=N\gg 0$, we have
$d(h(\overline{\delta}_j),\gamma_0)\approx N$ for all $j$. This, in
turn, implies that $i_S(h(\delta_j),\gamma)\gtrapprox {N\over 2}$
and $i_{S'}(h(\delta),\gamma')\gtrapprox {N\over 2} k\gg 2k$. Since
$i_{S'}(h'(\delta),\gamma')=i_{S'}(h(\delta),\gamma')$, we have a
contradiction and $\delta\not=h'(\delta)$.

\begin{figure}[ht]
\begin{overpic}[height=7cm]{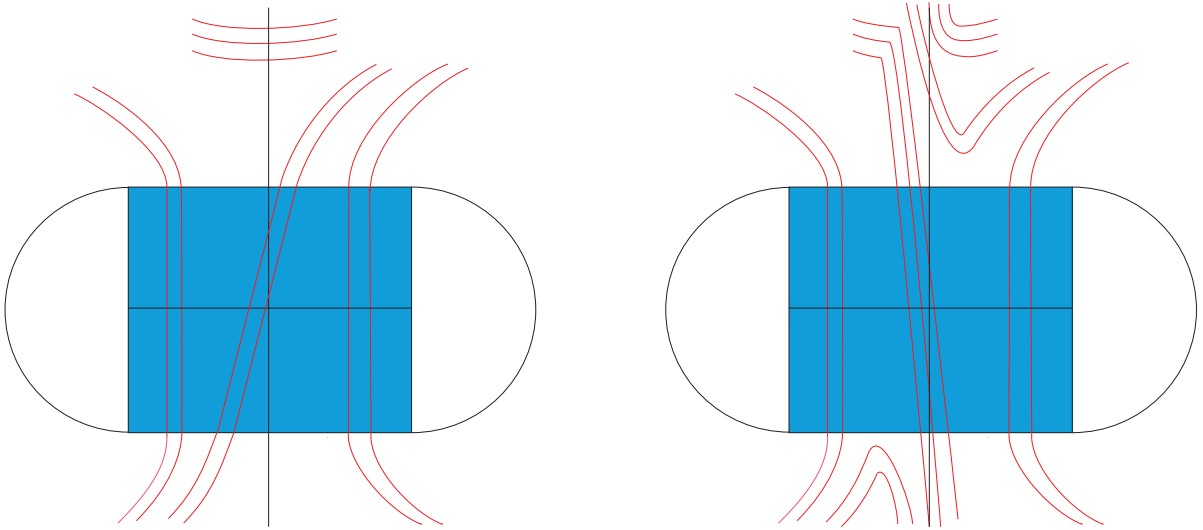}
\end{overpic}
\caption{The case where some $(h(\delta))_i'$ have positive slope.
The left-hand diagram depicts $h(\delta)$ and the right-hand diagram
depicts $R_{\gamma'}\circ h(\delta)$. The shaded region is the
1-handle $B$. In the picture, $k=6$, $m=2$, and $n=3$ for
$h(\delta)$.} \label{positive}
\end{figure}

Now that we know $h'$ is not reducible, it remains to show that $h'$
is not periodic, i.e., there is some $\delta$ such that
$(h')^j(\delta)\not=\delta$ for any $j\in \N$. For simplicity, take
$\delta=\bdry S$.  Normalizing with respect to $B$ and $\gamma$ as
before, $h'(\delta)=R_{\gamma'}(\delta)$ has $k_1=2$ intersections
with $a$ and $n_1=0$ intersections with $\gamma'$ away from $B$.
(Clearly, $h'(\delta)\not=\delta$.)  Suppose inductively that
$(h')^i(\delta)$ has $k_i$ intersections with $a$ and $n_i$
intersections with $\gamma'$ away from $B$, and that ${n_i\over
k_i}\ll 1$ and $k_i\rightarrow\infty$. Then, for each component
$((h')^i(\delta))_j$ of $(h')^i(\delta) \cap S$, we have
$d((\overline{(h')^i(\delta)})_j,\gamma_0)\approx 0$ and
$d(h((\overline{(h')^i(\delta)})_j),\gamma_0)\approx N$ for all $j$.
Hence $(h\circ (h')^i)(\delta)$ has $k_i$ intersections with $a$ and
$\gtrapprox {N\over 2}k_i$ intersections with $\gamma'$ away from
$B$. Using the same method as in the previous paragraphs,
$(h')^{i+1}(\delta)$ has  $\gtrapprox{N\over 2}k_i$ intersections
with $a$ and at most $k_i$ intersections with $\gamma'$ away from
$B$. Hence ${n_{i+1}\over k_{i+1}}\ll 1$, $k_{i+1}> k_i$, and
$(h')^i(\delta)\not= \delta$ for all $i\in\N$.

\s\n {\bf Case 2: $\bdry S$ has two components.}  Suppose that
$\bdry S$ has two components.

Let $\gamma$ be an arc which connects the two components of $\bdry
S$.  We assign to $\gamma$ a closed curve $\alpha_\gamma$ as
follows: Take a pair-of-pants neighborhood $N\subset S$ of
$\gamma\cup \bdry S$. Then let $\alpha=\alpha_\gamma$ be the
component of $\bdry N$ which is not a subset of $\bdry S$. On the
other hand, given a closed curve $\alpha$ which separates off a
pair-of-pants $N$, we can recover $\gamma$, provided we allow the
endpoints of $\gamma$ to freely move on $\bdry S$.

We explain how to pick a suitable arc of stabilization $\gamma$
which connects the two components of $\bdry S$.  Let $\gamma_0$ be
any arc which connects the two components of $\bdry S$, and
$\alpha_0=\alpha_{\gamma_0}$. Let $\mu$ be a stable lamination of a
pseudo-Anosov $g$, where $h$ is not freely homotopic to $g$. Then
the sequence of iterates $g^i(\alpha_0)$ converges to $\mu$ as
$i\rightarrow \infty$. Hence, by Lemma~\ref{lemma:findgamma}, there
exists $\alpha=g^n(\alpha_0)$, $n\gg 0$, so that $d(\alpha,
h(\alpha))=N\gg 0$. Now let $\gamma$ be the arc for which
$\alpha_\gamma=\alpha$. Then $i_S(\gamma,h(\gamma))$ and
$i_S(\alpha,h(\alpha))$ are roughly proportional, with a
proportionality factor of $4$. Here we are freely homotoping
$\gamma$ and $h(\gamma)$ along $\bdry S$ so their minimal
intersection number is realized.

Once we have picked $\gamma$ so that $\alpha_\gamma$ satisfies
$d(\alpha_\gamma,h(\alpha_\gamma))=N$, the rest of the proof is
almost identical to the case where $\bdry S$ has one boundary
component. The only difference is the definition of
$\overline{\delta}_i$, when $\delta_i$ has endpoints on the distinct
components of $\bdry S$.  In this case we set $\overline{\delta}_i=
\alpha_{\delta_i}$, as defined above.

\s\n This completes the proof of Theorem~\ref{thm:stabilization}.
\hfill$\Box$

\s Using the same techniques we can prove the following:

\begin{cor}
If $\bdry S$ is connected, $h$ is pseudo-Anosov, and the fractional
Dehn twist coefficient $c>2$, then any elementary stabilization
along a non-boundary-parallel arc $\gamma\subset S$ is
pseudo-Anosov.
\end{cor}

The {\em fractional Dehn twist coefficient $c$} of a pseudo-Anosov
$h$ is defined as follows:  Let $H:S\times[0,1]\rightarrow S$ be the
free homotopy from $h(x)=H(x,0)$ to its pseudo-Anosov representative
$\psi(x)= H(x,1)$.  Define $\beta:\bdry S\times[0,1]\rightarrow
\bdry S\times[0,1]$ by sending $(x,t)\mapsto (H(x,t),t)$, i.e.,
$\beta$ is the trace of the isotopy $H$ along $\bdry S$. If we
choose an oriented identification $\bdry S\simeq \R/\Z$, then we can
lift $\beta$ to $\tilde\beta: \R\times[0,1]\rightarrow
\R\times[0,1]$ and set $f(x)=\tilde\beta(x,1)-\tilde\beta(x,0)+x$.
The {\em fractional Dehn twist coefficient} $c\in\Q$ is the rotation
number of $f$, i.e., $$c=\lim_{n\rightarrow \infty} {f^n(x)-x\over
n},$$ for any $x$.

The inequality $c>2$ is certainly not optimal, but we will not
pursue the optimal constant here.

\begin{proof}
We use the same notation as in the proof of
Theorem~\ref{thm:stabilization}, and highlight only the differences.
If $\delta\subset S$, then $h(\delta)\not= \delta$ since $h$ is
pseudo-Anosov.  If $h(\delta)$ does not intersect $\gamma'$, then
$h'(\delta)=h(\delta)\not=\delta$.  Otherwise,
$i_{S'}(h'(\delta),a)>0$ whereas $i_{S'}(\delta,a)=0$.

If $\delta\not\subset S$, then we normalize $\delta$ with respect to
$B$ as in the proof of Theorem~\ref{thm:stabilization} and consider
the arcs $\delta_i$ of $S$; the $\delta_i$ are
non-boundary-parallel. Let $\pi:\widetilde S\rightarrow S$ be the
universal covering map. Fix a connected component of $\pi^{-1}(\bdry
S)$, which we call $L$. Let $\widetilde\gamma_j$, $j\in \Z$, be the
preimages of $\gamma$ which have an endpoint on $L$. If $p$ is an
endpoint of $\delta_i$, denote by $n(\delta_i,p)$ the geometric
intersection number, relative to the endpoints, of $\cup_j
\widetilde\gamma_j$ and a lift of $\delta_i$ which has an endpoint
on $\pi^{-1}(p)\cap L$.

Next, given the pair $(\delta_i,p)$ and a component of $\bdry S\cap
\bdry B$ which contains $p$, we define what it means for
$(\delta_i,p)$ to be to the left or to the right of $\gamma$. Using
the boundary orientation for $\bdry S$, if (the relevant component
of) $\bdry S\cap \bdry B$ intersects $p$ before it intersects
$\gamma$, then $(\delta_i,p)$ is {\em to the left of $\gamma$};
otherwise, $(\delta_i,p)$ is {\em to the right of $\gamma$}.
(Remember that the $\delta_i$ are normalized with respect to
$\gamma'$.)

If $(\delta_i,p)$ is to the right of $\gamma$, then we claim that
$n((h(\delta))_i,q)\geq 4$, where $q$ is an endpoint of
$(h(\delta))_i$. To see this, let $g=(R_{\bdry S})^2$. The arc
$g(\delta_i)$ satisfies $n(g(\delta_i),p)=4$.  Since $c>2$, $h\circ
g^{-1}$ is pseudo-Anosov with fractional Dehn twist coefficient
$c-2>0$, and hence is right-veering. Therefore $h(\delta_i)$ is to
the right of $g(\delta_i)$, and $n(h(\delta_i),p)\geq
n(g(\delta_i,p))=4$.  When $(\delta_i,p)$ is to the right of
$\gamma$, we may take $(h(\delta))_i=h(\delta_i)$ and $p=q$. This
proves $n((h(\delta))_i,q)\geq 4$.  On the other hand, if
$(\delta_i,p)$ is to the left of $\gamma$, then a similar
calculation yields
\begin{equation}\label{equation:sum}
n((h(\delta))_i,q)+n(\delta_i,p) > 3.
\end{equation}
In either case, Equation~\ref{equation:sum} is satisfied. This means
that, if $k$ is the number of components of $S\cap \delta$ and
$m,m'$ are the number of intersections of $\delta$, $h(\delta)$ with
$\gamma$ outside of $B$, then
$$m+m'\geq \sum ( n((h(\delta))_i,q)+n(\delta_i,p))> 6k.$$
(Remember that each $\delta_i$ has two endpoints.) In order for
$i_{S'}(\delta,a)=i_{S'}(h'(\delta),a)$, we require $m'\leq k$. This
implies that $m>5k$, which contradicts
$i_{S'}(\delta,\gamma')=i_{S'}(h'(\delta),\gamma')$.

We have shown that $h'$ is not reducible.  To show that $h'$ is not
periodic, consider $\delta=\bdry S$.  One easily verifies that the
number of intersections with $a$ increases with each iterate of
$h'$.
\end{proof}

\s\n {\em Acknowledgements.} We thank Yair Minsky for helpful e-mail
correspondence.

\end{document}